\documentclass[journal]{IEEEtran}

\IEEEoverridecommandlockouts                              

\usepackage{changes}
\usepackage{authblk}
\usepackage{amsfonts}
\usepackage{amsthm}
\usepackage{amsmath,amssymb}
\usepackage{bm}
\usepackage{hyperref,cite}
\usepackage{cleveref}
\usepackage{color}
\usepackage{algorithm}
\usepackage{algorithmic}
\let\labelindent\relax
\usepackage{enumitem}
\usepackage{float}
\usepackage{graphicx}
\usepackage{subfigure}
\usepackage{booktabs}
\usepackage{threeparttable}
\usepackage{hhline}
\usepackage{scalerel}
\usepackage{ulem}
\usepackage{tikz}
\usepackage{multirow}

\newtheorem{theorem}{Theorem}
\newtheorem{assumption}{Assumption}
\newtheorem{remark}{Remark}
\newtheorem{lemma}{Lemma}

\newtheorem{definition}{Definition}
 
\newcommand{\diag}{\mathop{\rm diag}\nolimits}

\newcommand{\rr}{{\mathbb R}}
\newcommand{\ba}[1]{\begin{array}{#1}}
\newcommand{\ea}{\end{array}}

\newcommand*{\pdot}{\mathbin{\scalerel*{\boldsymbol\odot}{\circ}}}

\begin{document}
\title{\LARGE \bf
Arbitrarily Small Execution-Time Certificate: \\ What was Missed in Analog Optimization}
\author{Liang Wu$^{1}$, Ambrose Adegbege$^{2}$, Yongduan Song$^{3}$, Richard D. Braatz$^{1}$
\thanks{This research was supported by the U.S. Food and Drug Administration under the FDA BAA-22-00123 program, Award Number 75F40122C00200. The authors thank Prof.\ Kunal Garg of Arizona State University for discussing Lemma \ref{lemma_non_negative}.\\
$^{1}$Massachusetts Institute of Technology, Cambridge, MA 02139, USA (email:{\tt\small \{liangwu,braatz\}@mit.edu}); $^{2}$The College of New Jersey, Ewing, NJ 08618 USA (email: {\tt\small adegbega@tcnj.edu});
$^{3}$Chongqing University, Chongqing 400044, P. R. China (email:{\tt\small ydsong@cqu.edu.cn}).
}
}

\maketitle

\begin{abstract}
\textit{Numerical optimization} (solving optimization problems using digital computers) currently dominates but has three major drawbacks: high energy consumption, poor scalability, and lack of an execution time certificate. To address these challenges, this article explores the recent resurgence of analog computers, proposing a novel paradigm of arbitrarily small execution-time-certified \textit{analog optimization} (solving optimization problems via analog computers). To achieve ultra-low energy consumption, this paradigm transforms optimization problems into ordinary differential equations (ODEs) and leverages the ability of analog computers to naturally solve ODEs (no need for time discretization) in physically real time. However, this transformation can fail if the optimization problem, such as the general convex nonlinear programs (NLPs) considered in this article, has no feasible solution. To avoid transformation failure and enable infeasibility detection, we introduce the homogeneous monotone complementarity problem formulation for convex NLPs. To achieve scalability and an execution time certificate, this paper introduces the Newton-based fixed-time-stable scheme for the transformed ODE, whose settling time $T_p$ can be prescribed by choosing the ODE's time coefficient as $k=\frac{\pi}{2T_p}$. This equation certifies that the settling time (execution time) is independent of the dimension of the optimization problems and can be arbitrarily small if the analog computer allows.
\end{abstract}


\section{Introduction}
Computation is vital. Though nearly all modern computers are digital (using discrete $0,1$ values), analog computers, once the most powerful, use continuous physical quantities (mechanical, electrical, etc.) to process information, as shown in Fig. \ref{fig_history}. 
\begin{figure}[!htbp]
    \centering
    \includegraphics[width=1.0\linewidth]{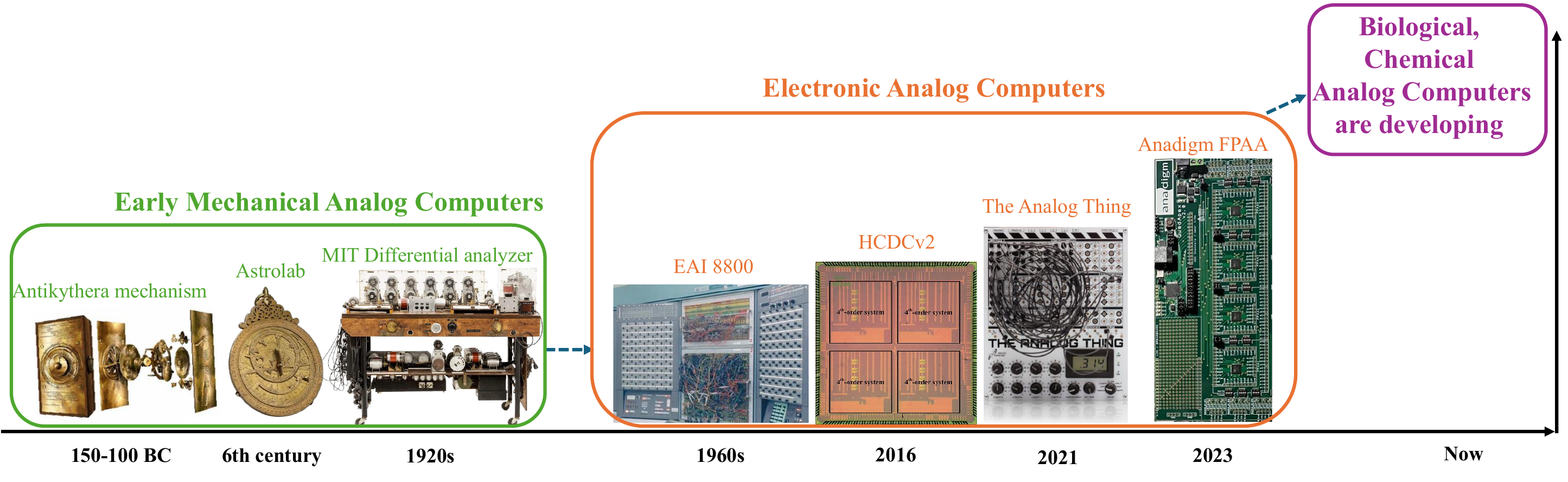}
    
    \vspace{-0.3cm}
    \caption{History development of analog computers.}\label{fig_history}
\end{figure}

Our world is continuous, and analog computing is more natural and highly energy-efficient \cite{sarpeshkar1998analog}. The 
low-power advantage of analog computing can easily be illustrated by three application examples: 
\textit{(i) vector-matrix-multiplication} (VMM, $x=Ab$). Digital computers require significant power to drive transistors for representing bits of a number and perform approximately 1,000 one-bit operations to multiply two 32-bit numbers. In contrast, an electrical analog computer can perform multiplication using voltage and conductance governed by Ohm’s and Kirchhoff’s laws \cite{li2018analogue}, offering vastly superior energy efficiency. \textit{(ii) solving systems of linear equations} ($Ax=b$). An electrical analog computer treats solving a system of linear equations in the same way as VMM, only by adding a negative feedback circuit \cite{sun2019solving}. However, digital computers have $O(n^3)$ and $O(n^2)$ time complexity in solving a system of linear equations and performing a matrix-vector multiplication, respectively. \textit{(iii) solving ordinary differential equations} (ODEs). Initially, analog computers were created to solve ODEs as their continuous-time dynamics naturally mirrored those of differential systems (the 1920s' MIT Differential Analyzer as shown in Fig.\  1). For example, circuit elements such as capacitors can physically implement the integration operation, allowing the electric system's voltage dynamics to represent the ODE of interest directly.

Today, several factors are driving a revival of analog computing, especially in the AI era. Advances have improved the stability and programmability of analog computers (e.g., The Analog Thing and Anadigm FPAA in Fig.\ \ref{fig_history}). At the same time, the growing energy demands of AI have renewed interest in analog computing, since core operations like VMM can be performed much more efficiently on analog hardware such as crossbars \cite{li2018analogue} than on digital computers. Furthermore, analog computers are increasingly being used to solve a broader range of scientific computing tasks \cite{cowan2005vlsi,maclennan2007review,song2024programming}, including \textit{partial differential equations} (PDEs) \cite{huang2017hybrid,malavipathirana2021fast,miscuglio2021approximate,rekeczky2002analogic}, \textit{optimization problems} \cite{vichik2016stability,vichik2014fast,vichik2014solving}, and real-time \textit{optimization-based control applications} \cite{adegbege2024analog,bena2023hybrid,levenson2016analog}. The underlying principle involves transforming PDEs and optimization problems into equivalent ODEs, which analog computers can naturally solve in real time and ultra-low energy, whereas digital computers require time-discretization schemes and consume significantly more energy. While precise values depend on technology node and implementation, one may conceptually anticipate analog integrator or cross-bar time-constants in the micro- to nano-second range ($10^{-6}\sim10^{-9}$ s) vs digital algorithm iterations in the millisecond to microsecond range ($10^{-3}\sim10^{-6}$ s), and energy per operation reductions of two to four orders of magnitude under favorable conditions.

This article aims to explore the future direction of solving optimization problems via \textit{analog optimization} (using analog computers). Numerical optimization algorithms, such as first-order methods, active-set methods, interior-point methods, and sequential quadratic programming methods, currently dominate in solving optimization problems. In addition to high energy consumption, the execution time of numerical optimization algorithms is not scalable (increases superlinearly as the dimension of optimization problems increases), and is also challenging to certify (the iteration complexity of numerical optimization algorithms usually relies on the distance between the initial point and the optimal point, which is hard to know in advance for most cases, except for box-constrained problems where a conservative upper bound can be defined \cite{wu2024direct}). To offer an appealing alternative, this article proposes a novel analog optimization paradigm with low energy consumption, scalability, and execution-time certificate.

In \textit{analog optimization}, the primary concern is the final equilibrium state of the ODE, which corresponds to the optimal solution of the original problem, rather than the entire evolution trajectory. If permitted by the analog hardware, indefinitely increasing the time-scaling coefficient of the ODE would accelerate analog optimization arbitrarily. This article builds the fundamental theory of analyzing when the transformed ODE converges to the equilibrium to offer an \textit{arbitrarily small execution-time certificate}, which was missed and overlooked in \textit{analog optimization} research.

\textbf{Contributions.} This article, for the first time, proposes a systematic methodology that transforms general constrained convex nonlinear programming (NLP) problems to ODEs whose settling time $T_p$ can be prescribed by setting the ODE's time coefficient to $k=\frac{\pi}{2T_p}$. As shown in Fig.\  \ref{fig_method}, it includes
\begin{enumerate}
    \item[1)] the transformation: constrained convex NLP $\rightarrow$ homogeneous monotone complementarity problems (HMCP) that enable infeasibility detection;
    \item[2)] the transformation: HMCP $\rightarrow$ fixed-time-stable ODE, whose settling time can be prescribed independently of problem dimension;
    \item[3)] proposing the use of an analog solver to realize the fixed-time-stable ODE.
\end{enumerate}
Proofs detailing these transformations are provided. Then, this article proposes the paradigm of \textit{arbitrarily small execution-time-certified analog optimization} and argues that this has a promising future for solving optimization problems.
\begin{figure*}[!htbp]
    \centering
    \includegraphics[width=0.9\linewidth]{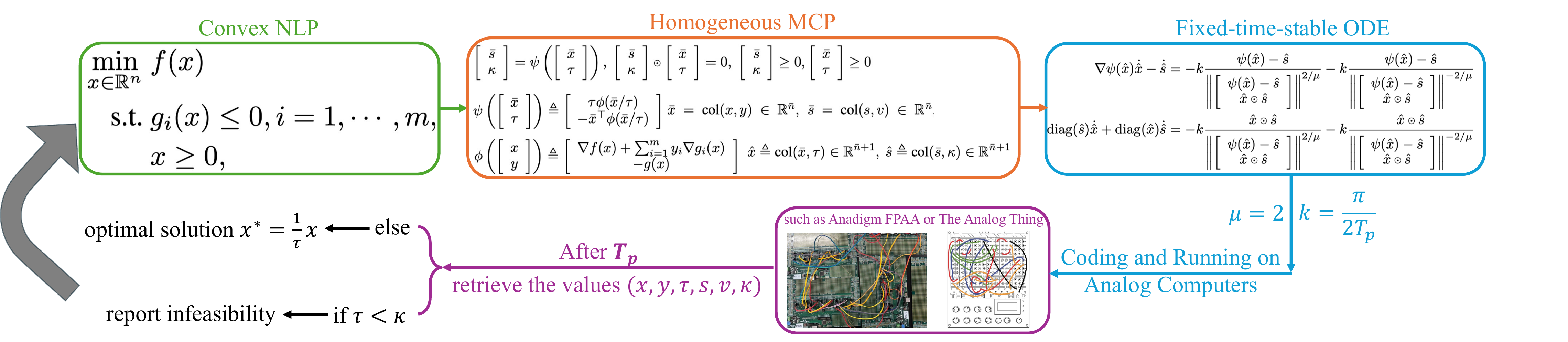}
    \caption{Arbitrarily small execution-time-certified analog optimization for solving convex NLPs.}
    \label{fig_method}
\end{figure*}

\textbf{Related work and method.} Transforming an optimization problem (especially for unconstrained optimization problems) into an ODE seems straightforward, as seen in works such as
\cite{boyd2024optimization,muehlebach2019dynamical,shi2022understanding,shi2019acceleration,wibisono2016variational,su2016differential}. They revealed deep links between optimization algorithmic design and ODEs, but the endpoint of all those works is how to design better numerical optimization algorithms to run on digital computers. Neuromorphic analog optimization platforms, including Ising machines, optical and spintronic networks, and other physics-based solvers \cite{mohseni2022ising}, seek to exploit continuous-time dynamics for energy-efficient computation. To certify the settling time, this article introduces the Newton-based fixed-time-stable scheme for the transformed ODE, whose settling time is independent of problem data and initial condition and thus supports the proposed paradigm of \textit{arbitrarily small execution-time-certified analog optimization}. Although the fixed-time-stable ODE scheme for optimization problems has been demonstrated in \cite{garg2020CAPPA,Garg2021fixed}, they are limited to constrained optimization problems with linear equality constraints (instead of our considered general convex NLP with inequality constraints). Worse still, they have a strong assumption: that the optimization problems have a feasible solution. To remove the feasibility assumption and achieve the infeasibility detection for a general convex NLP, this article adopts the homogeneous monotone complementarity problem formulation, which is then transformed into a fixed-time-stable ODE. Last, this article emphasizes that analog computers are perfect with this special ODE, as solving it on digital computers requires time discretization.

\section{Preliminaries: Fixed-time stable ODE}
Consider a continuous-time autonomous system,
\begin{equation}\label{eqn_continuous_dynamic}
    \dot x = F(x),
\end{equation}
where $x \in\rr^n, F:\rr^n\rightarrow\rr^n$.
\begin{assumption}
    The system \eqref{eqn_continuous_dynamic} has a unique equilibrium point $x^*$ satisfying $F(x^*)=0$. For any initial condition $x(0)=x_0\in\rr^n$, the trajectory of \eqref{eqn_continuous_dynamic} exists and is unique and continuous.
\end{assumption}

\begin{definition}
    (\textbf{Finite-time-stable}, see \cite[Def. 2]{polyakov2011nonlinear}): System \eqref{eqn_continuous_dynamic} is said to be {\em globally finite-time-stable} if it is globally asymptotically Lyapunov stable and reaches the equilibrium $x^*$ at some finite time $\forall x_0\in\rr^n\backslash\{x^*\}$, where the settling-time function $T(x_0)$ is bounded, i.e., for all $x_0\in\rr^n$, $\exists~T_{\max}(x_0)>0$ s.t.\ $T(x_0)\leq T_{\max}(x_0)$.
\end{definition}

\begin{lemma}
    (see \cite{Garg2021fixed,polyakov2011nonlinear}): the settling time $T(x_0)$ for the system \eqref{eqn_continuous_dynamic} can be bounded by
\begin{equation}\label{eqn_finite_time_stable_upper_bound}
        T(x_0)\leq T_{\max}(x_0)=\frac{V(x_0)^{1-\alpha}}{k(1-\alpha)},
    \end{equation}
    if there exists a positive-definite Lyapunov function $V\in\mathcal{C}^1(\mathcal{D},\rr)$ (where $\mathcal{D} \subset \mathbb{R}^n$ is a neighborhood of the equilibrium $x^*$) satisfying
    \begin{equation}
        \dot{V}(x)\leq -kV(x)^\alpha, \ \forall x\in\mathcal{D} \backslash \{x^*\},
    \end{equation}
    where the parameters $k>0$ and $\alpha\in(0,1)$.
\end{lemma}
The drawback of the \textit{finite-time-stable} concept is that, by \eqref{eqn_finite_time_stable_upper_bound}, the upper bound for the settling time $T(x_0)$ is dependent on the initial state $x_0$ and increases without bound when the magnitude of the initial state $\|x_0\|$ increases. To make the settling time independent from the initial state $x_0$, the concept of \textit{fixed-time-stable} is introduced.
\begin{definition}
    (\textbf{Fixed-time-stable}, see \cite{polyakov2011nonlinear}): The system \eqref{eqn_continuous_dynamic} is said to be {\em fixed-time-stable} if it is globally finite-time-stable and its settling-time function $T(x_0)$ is globally bounded, i.e., for all $x_0\in\rr^n$, $\exists T_{\max}\in(0,\infty)$ s.t.\ $T(x_0)\leq T_{\max}$.
\end{definition}

\begin{lemma}\label{lemma_fixed_time_stable}
    (see \cite{polyakov2011nonlinear,Garg2021fixed}): the settling time $T(x_0)$ for the system \eqref{eqn_continuous_dynamic} can be globally bounded by
    \begin{equation}\label{eqn_T_max_1}
       T(x_0)\leq T_{\max}=\frac{1}{k_1(1-\alpha_1)}+\frac{1}{k_2(\alpha_2-1)}. 
    \end{equation}
    if there exists a positive-definite Lyapunov function $V\in\mathcal{C}^1(\mathcal{D},\rr)$ (where $\mathcal{D} \subset \mathbb{R}^n$ is a neighborhood of the equilibrium) satisfying
\begin{equation}\label{eqn_Lyapunov_function_fixed_time_stable}
        \dot{V}(x)\leq -k_1V(x)^{\alpha_1}-k_2V(x)^{\alpha_2}, \ \forall x\in\mathcal{D} \backslash \{x^*\},
    \end{equation}
    where the parameters $k_1>0$, $k_2>0$, $0<\alpha_1<1$, and $\alpha_2>1$.
\end{lemma}
Now the upper bound $T_{\max}$ is independent of the initial point $x_0$, which makes it appealing for achieving the paradigm of arbitrarily small prescribed-time computing. Equation \eqref{eqn_T_max_1} for $T_{\max}$ is not tight; this article adopts a more accurate estimate.
\begin{lemma}\label{lemma_fixed_time_stable_simplified}
(see \cite[Lemma 2]{parsegov2012nonlinear}): If there exists a Lyapunov function $V$ satisfying \eqref{eqn_Lyapunov_function_fixed_time_stable} with 
\[
k_1=k_2=k,\quad\alpha_1=1-\frac{1}{2\mu},\quad\alpha_2=1+\frac{1}{2\mu},\quad\mu>1,
\]
 then the settling time $T(x_0)$ for the system \eqref{eqn_continuous_dynamic} can be globally bounded by
\begin{equation}\label{eqn_prescribed_time}
         T(x_0)\leq T_{\max} = \frac{\mu\pi}{\sqrt{k_1k_2}}=\frac{\mu\pi}{k}.
    \end{equation}
\end{lemma}

\section{Preliminaries: Unconstrained optimization to fixed-time stable ODE}
For an unconstrained optimization problem,
\begin{equation}\label{eqn_unconstrained_problem}
\min_{x\in\mathbb{R}^n} f(x)
\end{equation}
where the function $f: \mathbb{R}^n\rightarrow R$ is convex and differentiable, a point $x^*$ is the global optimal point of $f$ if and only if $\nabla f(x^*) = 0$  (see \cite{boyd2004convex}). Furthermore, if $f$ is strongly convex, then $x^*$ is unique. A straightforward approach to convert \eqref{eqn_unconstrained_problem} into an ODE is given by $\dot{x} = -\nabla f(x)$; however, this method does not guarantee an upper bound on the settling time.

\subsection{Gradient-based fixed-time stable ODE}
Transforming \eqref{eqn_unconstrained_problem} into a \textit{fixed-time-stable} ODE, one such approach is the gradient-based scheme,
\begin{equation}\label{eqn_fixed_time_stable_ODE_for_unconstrained}
    \dot{x} = -k\frac{\nabla f(x)}{\|\nabla f(x)\|^{2/\mu}}-k\frac{\nabla f(x)}{\|\nabla f(x)\|^{-2/\mu}},
\end{equation}
where $\mu>1$ and $k>0$.
\begin{lemma}\label{lemma_gradient_fixed_time_stalbe}
    Assume that the function $f(x)$ of \eqref{eqn_unconstrained_problem} is $m_f$-strongly convex, namely that,
    \begin{equation}\label{eqn_strongly_convex}
        f(y)\geq f(x)+\nabla f(x)^\top(y-x)+\frac{m_f}{2}\|y-x\|^2.
\end{equation}
    Then the solution of ODE \eqref{eqn_fixed_time_stable_ODE_for_unconstrained} exists and is unique for all $x(0)\in\rr^n$. Furthermore, given a desired prescribed settling time $T_p$ and let
\begin{equation}\label{eqn_graident_k}
k=\frac{\mu\pi}{4m_fT_p},
\end{equation}
    then the trajectories of ODE \eqref{eqn_fixed_time_stable_ODE_for_unconstrained} is \textit{fixed-time-stable} to the optimal solution $x^*$ of \eqref{eqn_unconstrained_problem} with the settling time $T\leq T_p$.
\end{lemma}
\begin{proof}
    It it clear that the $m_f$-strongly convex condition \eqref{eqn_strongly_convex} implies the Polyak-Łojasiewicz (PL) inequality:
\begin{equation}\label{eqn_PL_inequality}
 \tfrac{1}{2}\|\nabla f(x)\|^2\geq m_f(f(x)-f(x^*)), \ \forall x\in\mathbb{R}^n,   
\end{equation}
and the quadratic growth condition:
\begin{equation}\label{eqn_quadratic_growth}
    f(x)-f(x^*)\geq\frac{m_f}{2}\|x-x^*\|^2.
\end{equation}
Now consider the Lyapunov function $V(x)=\tfrac{1}{2}(f(x)-f(x^*))^2$, which is radially unbounded according to the quadratic growth condition \eqref{eqn_quadratic_growth}. 

First, by Okamura's Uniqueness Theorem (see \cite[Thm.\ 3.15.1]{agarwal1993uniqueness}), proving the existence and uniqueness of a solution for ODE \eqref{eqn_fixed_time_stable_ODE_for_unconstrained} is equivalent to proving that \textit{i)} $V(x^*)=0$, \textit{ii)} $V(x)>0$ if $x\neq x^*$, \textit{iii)} $V(x)$ is locally Lipschitz, and \textit{iv)} $\dot{V}(x)\leq0$. By the definition of $V(x)$ and the fact that \eqref{eqn_unconstrained_problem} has a unique minimizer $x^*$ if $f(x)$ is $m_f$-strongly convex, it is clear that \textit{i)} and \textit{ii)} hold. 

To prove \textit{iii)}, which is that $|V(x)-V(y)|\leq L\|x-y\|$ for some constant $L>0$, note that the gradient of $V(x)$ is $\nabla V(x)=\frac{d}{dx}\!\left[\frac{1}{2}(f(x)-f(x^*))^2\right]=(f(x)-f(x^*))\nabla f(x)$, and $|V(x)-V(y)|=|\nabla V(z)^\top(x-y)|\leq\|\nabla V(z)\|\cdot\|x-y\|$ (where $z=(1-\alpha)x+\alpha y$ and $\alpha\in[0,1]$). Thus the proof is reduced to showing that $\|(f(z)-f(x^*))\nabla f(z)\|\leq L$ for some constant $L>0$, which is clearly true under the fact that $f(z)-f(x^*)$ and $\nabla f(z)$ are both bounded (as $f(x)$ is continuous and $\nabla f(x)$ is locally Lipschitz) when $z$ are in a compact set.

To prove \textit{iv)}, the time derivative $\dot{V}(x)$ along the trajectories of ODE \eqref{eqn_fixed_time_stable_ODE_for_unconstrained} is
    \[
    \begin{aligned}
    &\quad\dot{V}(x)=(f(x)-f(x^*))(\nabla f)^\top\dot{x}\\
    &=-k(f(x)-f(x^*))\|\nabla f\|^{2-\frac{2}{\mu}} -k(f(x)-f(x^*))\|\nabla f\|^{2+\frac{2}{\mu}}\\
    &\quad\text{(by the PL inequality \eqref{eqn_PL_inequality})} \\
    &\leq-k(2m_f)^{1-\frac{1}{\mu}}(f(x)-f(x^*))^{1+1-\frac{1}{\mu}}\\
    &\quad -k(2m_f)^{1+\frac{1}{\mu}}(f(x)-f(x^*))^{1+1+\frac{1}{\mu}}\\
    &\quad\text{(substituting $f(x)-f(x^*)=(2V)^{1/2}$ )} \\
    &=-k(2m_f)^{1-\frac{1}{\mu}}(2V)^{1-\frac{1}{2\mu}}-k(2m_f)^{1+\frac{1}{\mu}}(2V)^{1+\frac{1}{2\mu}}\\
    &<0,
    \end{aligned}
    \]
    which completes the first part of the proof.
    
Second, by Lemma \ref{lemma_fixed_time_stable_simplified}, the upper bound for the settling time is
\[
    T_{\max}=\frac{\mu\pi}{\sqrt{k(2m_f)^{1-\frac{1}{\mu}}(2)^{1-\frac{1}{2\mu}}k(2m_f)^{1+\frac{1}{\mu}}(2)^{1+\frac{1}{2\mu}}}} = \frac{\mu\pi}{4km_f},  
\]
and letting $k=\frac{\mu\pi}{4m_fT_p}$, the upper bound $T_{\max}$ is equal to the prescribed time $T_p$, which completes the second part of the proof.
\end{proof}
However, \eqref{eqn_graident_k} relies on the strongly-convex parameter $m_f$ (hard to be known in advance). To address that, the next subsection introduces the Newton-based \textit{fixed-time-stable} ODE scheme.

\subsection{Newton-based fixed-time stable ODE}
A simple Newton-based ODE scheme for \eqref{eqn_unconstrained_problem} is given by $\dot{x} = -\left(\nabla^2 f(x)\right)^{\!-1} \nabla f(x)$; however, it lacks a certificate for the upper bound on the settling time. A Newton-based \textit{fixed-time-stable} ODE to \eqref{eqn_unconstrained_problem} is
\begin{equation}\label{eqn_FT_Newton}
    \dot{x} = -\left( \nabla^2 f(x)\right)^{\!-1}\left(k\frac{\nabla f(x)}{\|\nabla f(x)\|^{2/\mu}}+k\frac{\nabla f(x)}{\|\nabla f(x)\|^{-2/\mu}}\right)
\end{equation}
where $\mu>1$ and $k>0$.
\begin{lemma}\label{lemma_Newton_fixed_time_stalbe}
 Assume that the function $f(x)$ of \eqref{eqn_unconstrained_problem} is a two-times continuously differentiable and strongly convex function on $\rr^n$, the Hessian $\nabla^2f(x)$ is invertible for all $x\in\rr^n$, and the norm of the gradient, $\|\nabla f\|$, is radially unbounded. Then, the solution of ODE \eqref{eqn_FT_Newton} exists and is unique for all $x(0)\in\rr^n$. Furthermore, given a desired prescribed settling time $T_p$ and let
\begin{equation}\label{eqn_Newton_k}
k=\frac{\mu\pi}{4T_p},
\end{equation}
    then the trajectories of \eqref{eqn_FT_Newton} is \textit{fixed-time-stable} to the optimal solution $x^*$ of \eqref{eqn_unconstrained_problem} with the settling time $T\leq T_p$.
\end{lemma}
\begin{proof}
    By introducing $z\triangleq\nabla f(x)$, which implies that $\dot{z}=\nabla^2f(x)\dot{x}$, ODE \eqref{eqn_FT_Newton} is equivalent to the ODE,
\begin{equation}\label{eqn_FT_Newton_equivalent}
    \dot{z}=-k\frac{z}{\|z\|^{2/\mu}}-k\frac{z}{\|z\|^{\textcolor{blue}{-}2/\mu}}.
\end{equation}
Thus, we turn to proving that ODE \eqref{eqn_FT_Newton_equivalent} has a unique solution and is \textit{fixed-time-stable}. Now, consider the Lyapunov function $V(z)=\tfrac{1}{2}\|z\|^2$, which is radially unbounded according to the assumption that $\|z\|=\|\nabla f(x)\|$ is radially unbounded.

First, by Okamura's Uniqueness Theorem (see \cite[Thm.\ 3.15.1]{agarwal1993uniqueness}), proving the existence and uniqueness of a solution for ODE \eqref{eqn_FT_Newton_equivalent} is equivalent to proving that \textit{i)} $V(z^*)=0$ (where $z^*=\nabla f(x^*)$), \textit{ii)} $V(z)>0$ if $z\neq z^*$, \textit{iii)} $V(z)$ is locally Lipschitz, and \textit{iv)} $\dot{V}(z)\leq0$. By the definition of $V(z)$ and the fact that $z^*$ is unique as $f(x)$ is two-times continuously differentiable and strongly convex, it is clear that \textit{i)}, \textit{ii)}, and \textit{iii)} hold.

To prove \textit{iv)}, the time derivative $\dot{V}(z)$ along the trajectories of ODE \eqref{eqn_FT_Newton_equivalent} is
\[
\begin{aligned}
\dot{V}(z)&=z^\top \dot{z}=-k\|z\|^{2-2/\mu}-k\|z\|^{2+2/\mu}\\
&= -k (2V)^{1-\frac{1}{\mu}}-k (2V)^{1+\frac{1}{\mu}}<0,    
\end{aligned}
\]
which completes the first part of the proof.

Second, by Lemma \ref{lemma_fixed_time_stable_simplified}, the upper bound for the settling time is
\[
T_{\max}=\frac{\frac{\mu}{2}\pi}{\sqrt{k(2)^{1-\frac{1}{\mu}}k(2)^{1+\frac{1}{\mu}}}} = \frac{\mu\pi}{4k}
\]
and letting $k=\frac{\mu\pi}{4T_p}$, the upper bound $T_{\max}$ is equal to the prescribed time $T_p$, which completes the second part of the proof.
\end{proof}
By comparing \eqref{eqn_graident_k} and \eqref{eqn_Newton_k}, the Newton-based \textit{fixed-time-stable} ODE \eqref{eqn_FT_Newton} is independent of the problem data, and thus can prescribe an arbitrarily small settling time if the analog hardware allows. The above methodologies are limited to unconstrained optimization problems; one contribution of this article is to extend them to general convex NLPs, see the next section.

\section{Convex nonlinear programming to fixed-time-stable ODE}\label{sec_transformation}
This paper considers a convex NLP,
\begin{equation}\label{eqn_NLP}
\begin{aligned}
\min_{x\in\rr^n}&~f(x)\\
\text{s.t.}&~ g_i(x)\leq0,\ \ i=1,\cdots{},m,\\
        &~x\geq0,
\end{aligned}
\end{equation}
where $f(x):\rr^{n}\rightarrow\rr$ and $g_i(x):\rr^{n}\rightarrow\rr,i=1,\cdots{},m$ are all two-times continuously differentiable convex functions on $\rr^n_+$. That is, their Hessian matrices $\nabla^2f(x)$ and $\nabla^2g_i(x), i=1,\cdots{},m$ are symmetric positive semidefinite.  Any general convex NLP formulation (such as $\min~f(x),~\text{s.t.}~g_i(x)\leq0,i=1,\cdots{},m$) can always be transformed into the convex NLP \eqref{eqn_NLP} by using $x=x^+-x^-$ (where $x^+,x^-\geq0$ denote the positive and negative part of $x$, respectively). This representation is necessary because analog hardware is limited to supporting only nonnegative variables, for example, voltages are inherently nonnegative.
\begin{remark}
 Two special cases of convex NLP \eqref{eqn_NLP}, are linear programming (LP, where $f(x)=c^\top x, \ g(x)=Ax-b$) and quadratic programming (QP, where $f(x)=\frac{1}{2}x^\top Qx+c^\top x, \ g(x)=Ax-b$ with $Q=Q^\top\succeq0$).   
\end{remark}

The Karush–Kuhn–Tucker (KKT) condition \cite[Ch.\ 5]{boyd2004convex} for convex NLP \eqref{eqn_NLP} can be formulated as
\begin{equation}\label{eqn_MCP}
\begin{aligned}
&\left[\begin{array}{c}
s \\            v \end{array} \right]  = \phi\!\left(\left[\begin{array}{c}
    x  \\       y 
\end{array}\right]\right)\!, ~\left[\begin{array}{c}
 s  \\
 v        \end{array}\right] \pdot \left[\begin{array}{c}
x  \\           y   \end{array}\right]=0, \\
&\left[\begin{array}{c}
s  \\           v   \end{array}\right]\geq0,\ \  \left[\begin{array}{c}
             x  \\
             y 
        \end{array}\right]\geq0,
        \end{aligned}
\end{equation}
where
\[
\phi\left(\left[\begin{array}{c}
             x  \\
             y 
        \end{array}\right]\right)\triangleq\left[ \begin{array}{c}
     \nabla f(x)+\sum_{i=1}^my_i\nabla g_i(x) \\
     -g(x)
\end{array}\right]: \rr^{\Bar{n}}_{+}\rightarrow\rr^{\Bar{n}}.
\]
and the vectors $y\in\rr^m, s\in\rr^n$ denote the Lagrangian variable for the inequality constraint $g(x)=\mathrm{col}\left(g_i(x),\cdots{},g_m(x)\right)\leq0$ and $x\geq0$, respectively. The vector $v\in\rr^m$ denotes the slack variable: $v=-g(x)\in\rr^m$. The symbol $\pdot$ denotes the Hadamard product of two vectors.

To simplify the representation, denote $\Bar{x}=\text{col}(x,y)\in\rr^{\Bar{n}},~ \Bar{s}=\text{col}(s,v)\in\rr^{\Bar{n}}$, where $\bar{n}=n+m$.
\begin{definition}\label{def_MCP}
   A continuous mapping: $\phi:\rr_+^{\bar{n}}\rightarrow\rr^{\bar{n}}$ is said to be monotone on $\rr_+^{\bar{n}}$ if for all $\bar{x}^1,\bar{x}^2\in\rr^{\bar{n}}_+$, $(\bar{x}^1-\bar{x}^2)^\top(\phi(\bar{x}^1)-\phi(\bar{x}^2))\geq0$ holds. And the monotone complementarity problem (MCP) is to find vectors $\bar{x}\geq0,\bar{s}=\phi(\bar{x})\geq0$ such that $\bar{x}\pdot \bar{s}=0$.
\end{definition}

\begin{lemma}
    The Jacobian of $\phi$ is positive semidefinite on $\rr^{\bar{n}}_{+}$, and therefore the KKT \eqref{eqn_MCP} is an MCP by Definition \ref{def_MCP}.
\end{lemma}
\begin{proof}
First, the Jacobian matrix of $\phi$,
\[
\nabla\phi= \left[\begin{array}{cc}
    \nabla^2 f(x) + \sum_{i}^m y_i \nabla^2 g_i(x) &  \nabla g(x)^\top\\
      -\nabla g(x) & 0 
\end{array}\right],
\]
is positive semidefinite on $\rr^{n+m}_{+}$. This holds because, for an arbitrary $(\Delta x, \Delta y)$,
\[
\begin{aligned}
&\quad\left[\begin{array}{c}
     \Delta x  \\
     \Delta y 
\end{array} \right]^\top \nabla \phi\left[\begin{array}{c}
     \Delta x  \\
     \Delta y 
\end{array} \right] \\
&=\Delta x^\top\! \left( \nabla^2 f(x) + \sum_i^m y_i \nabla^2 g_i(x)\right)\!\Delta x\geq0,
\end{aligned}
\]
for all $x\geq0,y\geq0$, which completes the first part of the proof.

Second, for two arbitrary points $\Bar{x}_1,\Bar{x}_2\geq0$, define the function $\sigma(\delta) \triangleq (\Bar{x}_1-\Bar{x}_2)^\top(\phi(\Bar{x}_\delta)-\phi(\Bar{x}_2))$ with the scalar $\delta\in[0,1]$ and $\Bar{x}_\delta\triangleq\delta\Bar{x}_1+(1-\delta)\Bar{x}_2$. It is obvious that $\sigma(0) = 0$. The derivative $\frac{d\sigma(\delta)}{d\delta}= (\Bar{x}_1-\Bar{x}_2)^\top  \nabla \phi (\Bar{x}_1-\Bar{x}_2)\geq0$, so $\sigma(\delta)$ is non-decreasing. Therefore, $\sigma(1) \geq0$, which completes the second part of the proof.
\end{proof}

However, convex NLP \eqref{eqn_NLP} may be infeasible, and MCP \eqref{eqn_MCP} might also be infeasible, which will cause failure when transforming MCP \eqref{eqn_MCP} into an ODE, as the resulting ODE will not arrive at equilibrium. To avoid transformation failure and enable infeasibility detection, this article adopts a homogeneous model related to MCP \eqref{eqn_MCP} from   \cite{andersen1999homogeneous}.

\subsection{Homogeneous MCP}
By introducing two additional scalars $\tau\in\mathbb{R}$ and $\kappa\in\mathbb{R}$, a homogeneous model for MCP \eqref{eqn_MCP} is formulated as 
\begin{equation}\label{eqn_HMCP}
\begin{aligned} 
&\left[\begin{array}{c}
\Bar{s}  \\
    \kappa 
\end{array}\right]
 = \psi\! \left(\left[\begin{array}{c}
\bar{x}  \\
\tau 
\end{array} \right]\right),~\left[\begin{array}{c}
\Bar{s}  \\    \kappa 
\end{array} \right] \pdot \left[\begin{array}{c}
     \Bar{x}  \\
   \tau
\end{array} \right] =0,\\
&\left[\begin{array}{c}
\Bar{s}  \\
    \kappa
\end{array}\right]\geq0, \ \  \left[\begin{array}{c}
    \Bar{x}  \\
    \tau 
\end{array}\right]\geq0, 
\end{aligned}
\end{equation}
where
\begin{equation}\label{eqn_psi_definition}
\psi\left(\left[\begin{array}{c}
             \bar{x}  \\
             \tau 
        \end{array}\right]\right)\triangleq\left[\begin{array}{c}
\tau\phi(\Bar{x}/\tau)  \\
         -\Bar{x}^\top \phi(\Bar{x}/\tau)
    \end{array} \right]: \ \rr^{\Bar{n}+1}_{++}\rightarrow\rr^{\Bar{n}+1}.
\end{equation}
To simplify the representation, denote
\[
\hat{x}\triangleq \text{col}(\Bar{x},\tau)\in\rr^{\Bar{n}+1},\quad\hat{s}\triangleq \text{col}(\Bar{s},\kappa)\in\rr^{\Bar{n}+1}.
\]
\begin{lemma}\label{lemma_NLP_equal_MCP} 
The Jacobian $\nabla \psi(\hat{x})$  is positive semidefinite in $\rr^{\Bar{n}+1}_{++}$. Furthermore, $\psi:\rr^{\bar{n}+1}_{++}\rightarrow\rr^{\bar{n}+1}$ is a continuous monotone mapping on $\rr^{\bar{n}+1}_{++}$.
\end{lemma}
\begin{proof}
First, from the definition of $\psi(\hat{x})$, we have that
\[
\begin{aligned}
    &\nabla\psi=\\
    &\left[\begin{array}{@{}cc@{}}
\nabla\phi(\Bar{x}/\tau)  & \phi(\Bar{x}/\tau)- \nabla\phi(\Bar{x}/\tau)(\Bar{x}/\tau)\\
      -\phi(\Bar{x}/\tau)^\top- (\Bar{x}/\tau)^\top\nabla\phi(\Bar{x}/\tau) & (\Bar{x}/\tau)^\top\nabla\phi(\Bar{x}/\tau)(\Bar{x}/\tau)
    \end{array} \right]\!.      
\end{aligned}
\]
Then, under $(\Bar{x},\tau)>0$ and $\nabla \phi$ being positive semidefinite from Lemma \ref{lemma_NLP_equal_MCP}, we have that
\[
\begin{aligned}
&(d_{\Bar{x}},d_\tau)^\top \nabla \psi(d_{\Bar{x}},d_\tau)=d_{\Bar{x}}^\top\nabla\phi(\Bar{x}/\tau)d_{\Bar{x}}\\
& + d_{\Bar{x}}^\top\phi(\Bar{x}/\tau)d_\tau - d_{\Bar{x}}^\top\nabla\phi(\Bar{x}/\tau)(\Bar{x}/\tau)\\
&-d_{\tau}\phi(\Bar{x}/\tau)^\top d_{\Bar{x}}-d_{\tau}(\Bar{x}/\tau)^\top\nabla\phi(\Bar{x}/\tau)d_{\Bar{x}}\\
&+d_{\tau}^2(\Bar{x}/\tau)^\top\nabla\phi(\Bar{x}/\tau)(\Bar{x}/\tau)\\
&=\left(d_{\Bar{x}}-\frac{d_\tau}{\tau}\bar{x}\right)^{\!\!\top}\nabla\phi(\Bar{x}/\tau)\left(d_{\Bar{x}}-\frac{d_\tau}{\tau}\bar{x}\right)\geq0 
\end{aligned}
\]
for arbitrary $(d_{\Bar{x}},d_\tau)\in\rr^{\Bar{n}+1}$, which completes the first part of the proof.

Second, proving that $\psi$ is a continuous monotone mapping is similar to the proof of Lemma \ref{lemma_NLP_equal_MCP} (by $\psi(\hat{x})$ is positive semidefinite), which completes the second part of the proof.
\end{proof}
Thus, by Lemma \ref{lemma_NLP_equal_MCP}, the homogeneous model \eqref{eqn_HMCP} is also an MCP, hereinafter referred to as HMCP.

\subsection{Relationship between MCP and HMCP}
We first prove that HMCP \eqref{eqn_HMCP} always has an optimal solution in Lemma \ref{lemma_always_asymptotically_feasible}, then introduce the definition of maximal complementary solution, and state that obtaining a maximal complementary solution of HMCP \eqref{eqn_HMCP} can recover an optimal solution or report the infeasibility of MCP \eqref{eqn_MCP} in Lemma \ref{lemma_relationship_HMCP_MCP}.

\begin{lemma}\label{lemma_always_asymptotically_feasible}
HMCP \eqref{eqn_HMCP} is always asymptotically feasible. Furthermore, every asymptotically feasible solution of HMCP \eqref{eqn_HMCP} is an asymptotically complementary solution. In this case, by the definition, a feasible and complementary solution is an optimal solution.
\end{lemma} 
\begin{proof}
HMCP \eqref{eqn_HMCP} is said to be asymptotically feasible if and only if there exists positive and bounded $(\bar{x}^t,\tau^t,\bar{s}^t,\kappa^t)>0, \ t>0$ such that $\lim_{t\rightarrow\infty}\mathrm{col}(\bar{s}^t,\kappa^t)-\psi(\bar{x}^t,\tau^t)\rightarrow0$. An asymptotically feasible solution $(\bar{x}^t,\tau^t,\bar{s}^t,\kappa^t)$ of HMCP \eqref{eqn_HMCP} such that $(\bar{x}^t)^\top \bar{s}^t + \tau^t \kappa^t=0$
    is said to be an asymptotically complementary solution of HMCP \eqref{eqn_HMCP}.
    
Let $\bar{x}^t = (\frac{1}{2})^te, \tau^t = (\frac{1}{2})^te, \bar{s}^t = (\frac{1}{2})^te$, and $\kappa^t = (\frac{1}{2})^t$. Then, as $t\rightarrow\infty$,
    \[
\begin{aligned}
\left[\begin{array}{c}
\bar{s}^t  \\
\kappa^t
\end{array}\right]-
\psi\left(\left[\begin{array}{c}
     \bar{x}^t  \\
     \tau^t 
\end{array}\right]\right)=  \left(\frac{1}{2}\right)^{\!\!t}\left[\begin{array}{c}
            e-Me-q  \\
    1+e^\top M e+e^\top q
\end{array}\right]\!\rightarrow0,
\end{aligned}
    \]
which completes the first part of the proof.

For each $(\bar{x}^t,\tau^t,\bar{s}^t,\kappa^t)$ asymptotically feasible solution of HMCP \eqref{eqn_HMCP}, the complementary condition $(\bar{x}^t)^\top \bar{s}^t + \tau^t \kappa^t=0$ always holds by  \eqref{eqn_psi_definition}, which completes the second part of the proof.
\end{proof}

\begin{definition}
A complementary solution $(\bar{x}^*,\tau^*,\bar{s}^*,\kappa^*)$ for HMCP \eqref{eqn_HMCP} is said to be a {\em maximal complementary solution} such that the number of positive components in $(\bar{x}^*,\tau^*,\bar{s}^*,\kappa^*)$ is maximal.
\end{definition}

\begin{lemma}\label{lemma_relationship_HMCP_MCP}
(see \cite[Thm.\ 1]{andersen1999homogeneous}):  Let $(\bar{x}^*,\tau^*,\bar{s}^*,\kappa^*)$ be a maximal complementary solution for HMCP \eqref{eqn_HMCP},
\begin{itemize}
\item[i)] MCP \eqref{eqn_MCP} has a solution if and only if $\tau^*>0$. In this case, $\mathrm{col}(\bar{x}^*/\tau^*,\bar{s}^*/\tau^*)$ is a optimal solution for MCP \eqref{eqn_MCP};
    \item[ii)] MCP \eqref{eqn_MCP} is infeasible if and only if $\kappa^*>0$. In this case, $\mathrm{col}(\bar{x}^*/\kappa^*,\bar{s}^*/\kappa^*)$ is a certificate to prove infeasibility.
\end{itemize}
\end{lemma}
Thus, finding an optimal solution or detecting the infeasibility of convex NLP \eqref{eqn_NLP}, i.e. MCP \eqref{eqn_MCP}, is equivalent to finding a maximal complementary solution of HMCP \eqref{eqn_HMCP}. The following lemma, from the known central path theory in interior-point methods (IPMs), offers an approach to find a maximal complementary solution of HMCP \eqref{eqn_HMCP}.
\begin{lemma}\label{lemma_HLCP_central_path}(see \cite[ Thm.\ 2]{andersen1999homogeneous}):
For any $\theta>0$, starting from $(\bar{x}^0>0, \tau^0>0,\bar{s}^0> 0,\kappa^0>0)$, there is a unique strictly positive point $(\bar{x}(\theta),\tau(\theta),\bar{s}(\theta),\kappa(\theta))$ such that
\[
\begin{aligned}
\psi\left(\left[\begin{array}{c}
     \bar{x}(\theta)  \\
     \tau(\theta) 
\end{array}\right]\right) - \left[\begin{array}{c}
    \bar{s}(\theta) \\
     \kappa(\theta) 
\end{array}\right] &= \theta\! \left(\psi\left(\left[\begin{array}{c}
     \bar{x}^0 \\
     \tau^0  
\end{array}\right]\right)-\left[\begin{array}{c}
    \bar{s}^0 \\
    \kappa^0 
\end{array}\right]\right)\!, \\
\left[\begin{array}{c}
    \bar{s}(\theta) \pdot\bar{x}(\theta) \\
    \kappa(\theta) \tau(\theta)  
\end{array}\right] &= \theta\! \left[\begin{array}{c}
    \bar{s}^0 \pdot\bar{x}^0 \\
    \kappa^0 \tau^0  
\end{array}\right],   
\end{aligned}
\] 
Then,
\begin{itemize}
    \item[i)] For any $\theta>0$, the trajectory $(\bar{x}(\theta),\tau(\theta),\bar{s}(\theta),\kappa(\theta))$ is a continuous bounded trajectory;
    \item[ii)] When $\theta\rightarrow0$, any limit point $(\bar{x}(\theta),\tau(\theta),\bar{s}(\theta),\kappa(\theta))$ is a maximal complementary solution for HMCP \eqref{eqn_HMCP}. 
\end{itemize}
\end{lemma}
However, IPMs solve HMCP \eqref{eqn_HMCP} in the discrete domain and thus require procedures such as line search to ensure all iterates $\{(\bar{x}^k,\tau^k,\bar{s}^k,\kappa^k)\}$ ($k$ denotes the $k$th iterate) are positive, leading to low computational efficiency. Inspired by this, this article explores solving HMCP \eqref{eqn_HMCP} in continuous time to achieve positivity automatically (see Lemma \ref{lemma_non_negative} in the next subsection). That is, transforming  HMCP \eqref{eqn_HMCP} into an ODE allows us to approach the solution in a continuous-time domain. Moreover, this article introduces the Newton-based fixed-time-stable ODE scheme for this transformation.

\subsection{HMCP to fixed-time stable ODE}
This article converts HMCP \eqref{eqn_HMCP} into a Newton-based \textit{fixed-time-stable} ODE as
\begin{equation}\label{eqn_HMCP_fixed-time-stable_ODE}
\begin{aligned}
\quad\nabla \psi(\hat{x}) \dot{\hat{x}} - \dot{\hat{s}} &=-k\frac{\psi(\hat{x})-\hat{s}}{\left\|\left[\begin{array}{c}
\psi(\hat{x})-\hat{s} \\
\hat{x}\pdot \hat{s}
\end{array}\right]\right\|^{2/\mu}} \\
&\quad- k\frac{\psi(\hat{x})-\hat{s}}{\left\|\left[\begin{array}{c}
\psi(\hat{x})-\hat{s} \\
\hat{x}\pdot \hat{s}
\end{array}\right]\right\|^{-2/\mu}}\\
\diag(\hat{s}) \dot{\hat{x}} + \diag(\hat{x})\dot{\hat{s}} &=-k\frac{\hat{x}\pdot \hat{s}}{\left\|\left[\begin{array}{c}
\psi(\hat{x})-\hat{s} \\
\hat{x}\pdot \hat{s}
\end{array}\right]\right\|^{2/\mu}}\\
&\quad- k\frac{\hat{x}\pdot \hat{s}}{\left\|\left[\begin{array}{c}
\psi(\hat{x})-\hat{s} \\
\hat{x}\pdot \hat{s}
\end{array}\right]\right\|^{-2/\mu}} 
\end{aligned}
\end{equation}
where $\mu>1$ and $k>0$.

\begin{lemma}\label{lemma_uniqueness}
 Given an initial point $(\hat{x}(0),\hat{s}(0))\neq(\hat{x}^*,\hat{s}^*)$, the trajectory $(\hat{x}(t),\hat{s}(t))$ of ODE \eqref{eqn_HMCP_fixed-time-stable_ODE} exists and is unique for all $t\geq0$. Furthermore, for
\begin{equation}
    \mu=2,\quad k=\frac{\pi}{2T_p},
\end{equation}
the settling time of ODE \eqref{eqn_HMCP_fixed-time-stable_ODE} is bounded by $T_p$.
\end{lemma}
\begin{proof}
    First, by introducing 
    \[
    z_1 \triangleq \psi(\hat{x})-\hat{s}\in\rr^{\bar{n}+1},~z_2\triangleq \hat{x}\pdot \hat{s}\in\rr^{\bar{n}+1},
    \] 
we have $\dot{z}_1=\frac{d}{dt}\! \left(\psi(\hat{x})-\hat{s}\right)=\nabla \psi(\hat{x}) \dot{\hat{x}} - \dot{\hat{s}}$ and $\dot{z}_2=\frac{d}{dt}\! \left( \hat{x}\pdot \hat{s}\right)= \diag(\hat{s}) \dot{\hat{x}} + \diag(\hat{x})\dot{\hat{s}}$. Then, by letting $z\triangleq\mathrm{col}(z_1,z_2)$, ODE \eqref{eqn_HMCP_fixed-time-stable_ODE} is equivalently reformulated as
\begin{equation}\label{eqn_ODE_z_simplified}
\dot{z}=-k\frac{z}{\|z\|^{2/\mu}}-k\frac{z}{\|z\|^{-2/\mu}}.
    \end{equation}
    Thus, we turn to proving that ODE \eqref{eqn_ODE_z_simplified} has a unique solution and is \textit{fixed-time-stable}. Now, consider the Lyapunov function
    \[
    V(z)=\frac{1}{2}\|z\|^2,
    \]
    which is radially unbounded. Proving the existence and uniqueness of a solution for ODE \eqref{eqn_ODE_z_simplified} by Okamura's Uniqueness Theorem (see \cite[Thm.\ 3.15.1]{agarwal1993uniqueness}), is equivalent to proving that: \textit{i)} $V(0)=0$, \textit{ii)} $V(z)>0$ if $\|z\|\neq0$, \textit{iii)} $V(z)$ is locally Lipschitz, and \textit{iv)} $\dot{V}(z)\leq0$. By the definition of $V(z)$, it is clear that \textit{i)}, \textit{ii)}, and \textit{iii)} hold. To prove \textit{iv)}, the time derivative $\dot{V}(z)$ along the trajectories of ODE \eqref{eqn_ODE_z_simplified} is
    \begin{equation}
    \begin{aligned}
         \dot{V}&=z^\top \dot{z}=-k\|z\|^{2-2/\mu}-k\|z\|^{2+2/\mu}\\
         &= -k (2V)^{1-1/\mu} -k (2V)^{1+1/\mu}<0,   
    \end{aligned}
    \end{equation}
    which completes the first part of the proof. Then, by Lemma \ref{lemma_fixed_time_stable_simplified}, the upper bound for the settling time is 
    \[
T_{\max}=\frac{\frac{\mu}{2}\pi}{\sqrt{k(2)^{1-\frac{1}{\mu}}k(2)^{1+\frac{1}{\mu}}}} = \frac{\mu\pi}{4k}=T_p,
    \]
    which completes the second part of the proof.
\end{proof}

\begin{lemma}\label{lemma_non_negative}
    Given an initial point $(\hat{x}(0),\hat{s}(0))>0$, the trajectory $(\hat{x}(t),\hat{s}(t))$ of ODE \eqref{eqn_HMCP_fixed-time-stable_ODE} always satisfies $\hat{x}(t)\geq0, \hat{s}(t)\geq0$. 
\end{lemma}
\begin{proof}
The proof is by contradiction. Given an initial point $(\hat{x}(0),\hat{s}(0))>0$, we have that $z_2(0)=\hat{x}(0)\pdot\hat{s}(0)>0$. According to the Forward Invariance Theorem (Nagumo’s Theorem) of barrier functions (see   \cite{glotfelter2017nonsmooth}), $z_2(t)\geq0$ always holds if  $z_2(0)>0$, which implies that $\hat{x}(t)\geq0,\hat{s}(t)\geq0$ or $\hat{x}(t)\leq0,\hat{s}(t)\leq0$. 
     Let $\tau$ be defined as $\tau = \inf\{t\; \hat{x}(t) <0, \hat{s}(t) < 0\}$ so that for all $t\leq \tau$, $\hat{x}(t) \geq 0$ and $\hat{s}(t) \geq 0$. Now, for such a time instant $\tau$ to exist, it is necessary that the trajectory of $\hat{x}$ or $\hat{s}$ reach the origin and leave the origin so that either $\hat{x}$ or $\hat{s}$ can switch signs. As a result, once the trajectories of $z_2$ reach the origin, they remain at the origin. Thus, $\hat{x}(t)$ and $\hat{s}(t)$ cannot change the sign and, as a result, there exists no such $\tau$, which completes the proof.
\end{proof}

Combining Lemma \ref{lemma_relationship_HMCP_MCP}, \ref{lemma_HLCP_central_path}, \ref{lemma_uniqueness}, and \ref{lemma_non_negative} leads to the following theorem.
\begin{theorem}\label{thm_1}
    Given an initial point $(\hat{x}(0),\hat{s}(0))>0$, a prescribed settling time $T_p$, and choosing 
\begin{equation}
\mu=2,\quad k=\frac{\pi}{2T_p}
\end{equation}
    for ODE \eqref{eqn_HMCP_fixed-time-stable_ODE}, then its trajectory $(\hat{x}(t),\hat{s}(t))$ is fixed-time-stable to a maximal complementary solution $(\hat{x}^*,\hat{s}^*)$ of HMCP \eqref{eqn_HMCP} within the settling time $T_p$. Furthermore, retrieve the values of $(x,y,\tau,s,v,\kappa)$ at time $T_p$; if $\tau<\kappa$, report that convex NLP \eqref{eqn_NLP} is infeasible; otherwise, return the optimal solution $x^*=\frac{1}{\tau}x$ for  convex NLP \eqref{eqn_NLP}.
\end{theorem}

\section{Examples}\label{sec_examples}
The main contribution of this article is Thm.\ \ref{thm_1}, which supplements the long-neglected problem of analog optimization: how to prescribe the settling time of the transformed ODE, namely, to prescribe an arbitrarily small execution time for the original optimization problem. This article is the first part of a series of research, focusing on building the theoretical foundation and the ``soft-verification'': simulating the ODE to verify the correctness of Thm.\ \ref{thm_1}, which is illustrated in the following subsections for which we have made Julia codes publicly available at \url{https://github.com/liangwu2019/AnalogOptimization} (using the DFBDF() ODE solver, on a MacBook Pro with 2.7 GHz 4-core Intel Core i7 and 16 GB RAM). The second part (the ``hard-verification''), focusing on the analog hardware implementation, including The Analog Thing \cite{theanalogthing_docs}, Anadigm FPAA \cite{anadigm_techdocs}, and the tailored printed circuit board (PCB) for LP, QP, and convex NLP, is currently in progress.

\subsection{Infeasible LP, QP, and convex NLP examples}
To demonstrate the infeasibility-detection capability of our proposed homogeneous fixed-time-stable ODE formulation, we randomly generate 100 LPs ($\min c^\top x,\, \text{s.t.}\, Ax\geq b,x\geq0$), QPs ($\min \tfrac{1}{2}x^\top Qx+c^\top x,\,\text{s.t.}\, Ax\geq b,x\geq0$), and convex NLPs ($\min\,\sum_{i=1}^ne^{x_i},\,\text{s.t.}\,Ax\geq b,x\geq0$) (the problem dimension is $n=5,m=2$). To make these LPs, QPs, and convex NLPs infeasible, we add the contradictory constraint: $\sum_{i=1}^nx_i\leq-1$, namely $A\leftarrow[A;-\mathbf{1}_n^\top],b\leftarrow[b;1]$. The settling time is prescribed as $T_p=1$ seconds. By retrieving the values of $\tau,\kappa$ at $t=1$ seconds, all examples in three LP, QP, and convex NLP cases show that $\kappa>\tau$ and $\tau\rightarrow0$, which indicates that the infeasibility detection rates for all cases are $\mathbf{100\%}$. The trajectory of $\tau$ and $\kappa$ for a representative case is plotted in Fig.\  \ref{fig_result}(a).

\subsection{Feasible LP, QP, and convex NLP examples}
For the above random LP, QP, and convex NLP cases (all are feasible without the added contradictory constraint), two types of experiments are conducted: (1) we choose different prescribed settling time $T_p=1$, $0.8$, $0.6$, $0.4$, $0.2$, $0.1$ via setting the coefficient 
$k={\pi}/({2T_p})$; 
(2) we chose various initial conditions, $z^0=5e,z^0=10e,z^0=2e,z^0=40e, z^0=60e,z^0=80e$, under the same  prescribed settling time $T_p=1$. For the convex NLP examples, the results of Experiments 1 and 2 are plotted in Figs.\ \ref{fig_result}(b) and (c), respectively. The plotted results for LP and QP are similar and therefore omitted here. All of those validate that the settling time of our proposed ODE can be prescribed arbitrarily small and independent of the initial condition.
\begin{figure*}[!ht]\label{fig}
\begin{picture}(140,110)
\put(0,0){\includegraphics[width=60mm]{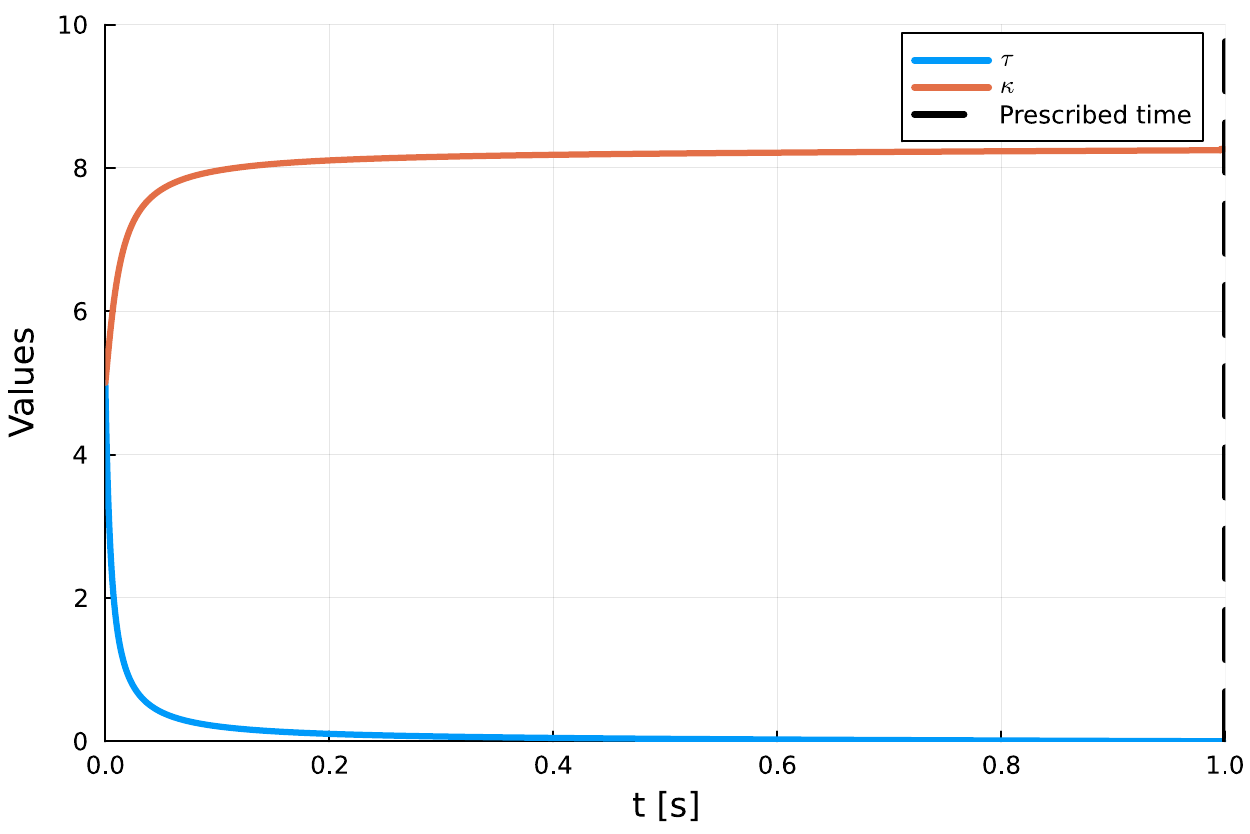}}
\put(180,0){\includegraphics[width=60mm]{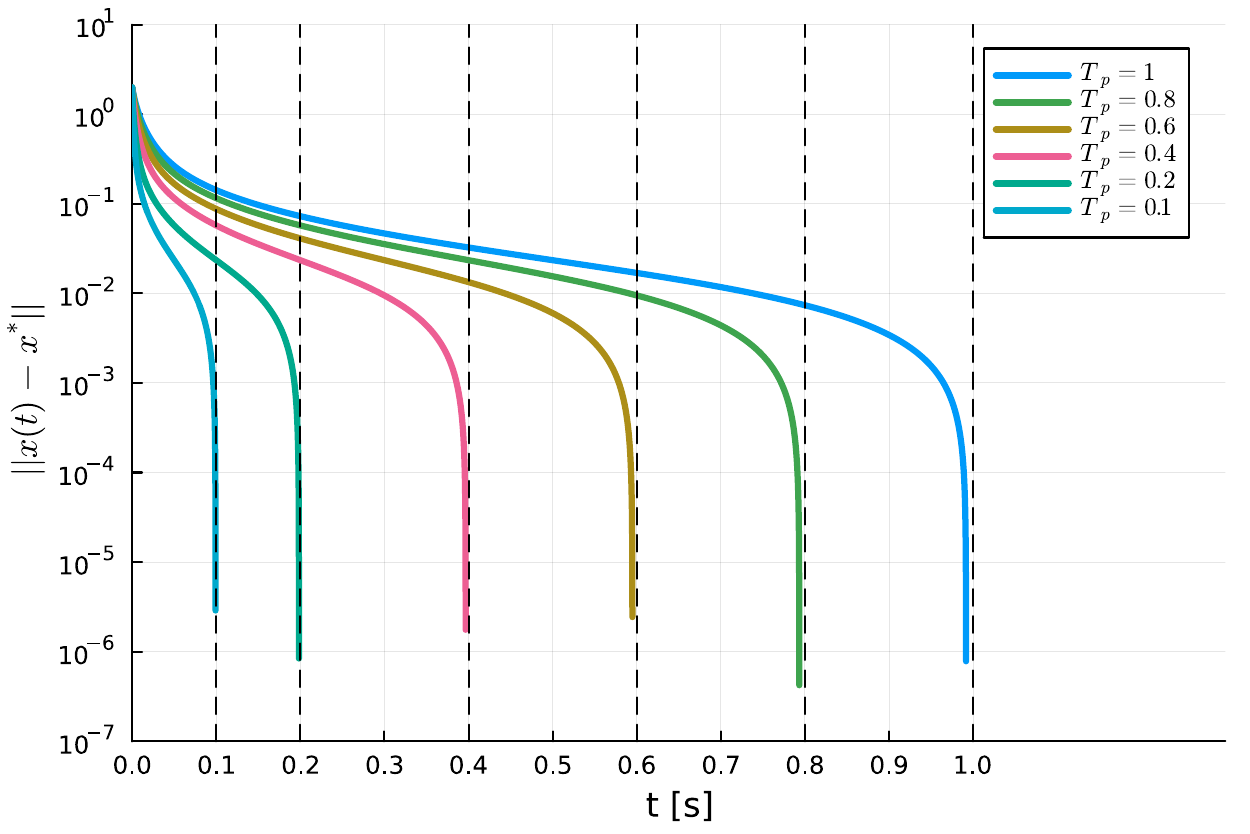}}
\put(350,0){\includegraphics[width=60mm]{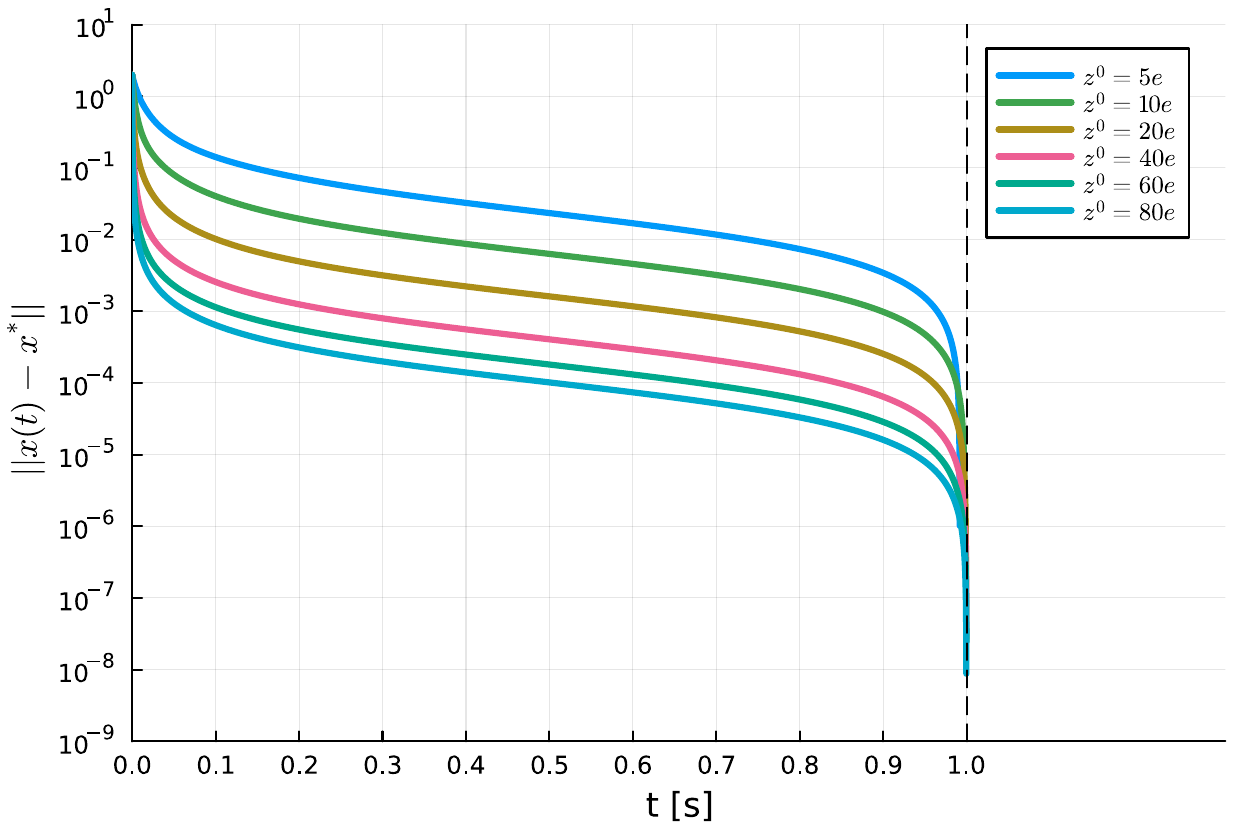}}
\put(85,-5){\makebox(0,0)[l]{\small (a)}}
\put(267,-5){\makebox(0,0)[l]{\small (b)}}
\put(437,-5){\makebox(0,0)[l]{\small (c)}}
\end{picture}
\caption{(a): The $\tau,\kappa$ trajectory of one Infeasible convex NLP example; (b) Experiment (1)'s trajectory$\|x(t)-x^*\|$ of one Feasible convex NLP example; (c): Experiment (2)'s $\|x(t)-x^*\|$ trajectory of one Feasible convex NLP example.}
\label{fig_result}
\end{figure*}

\section{Conclusion: limitations and future work}\label{sec_conclusion}
This article proposes a paradigm of \textit{arbitrarily small execution-time-certified analog optimization} for solving general constrained convex NLP problems, by integrating the HMCP formulation with a Newton-based fixed-time-stable ODE scheme. Compared to \textit{numerical optimization}, the proposed \textit{analog optimization} offers ultra-low energy consumption, scalability, and an execution time certificate. These features are valuable for real-time optimization-based control applications such as model predictive control (MPC) \cite{borrelli2017predictive}, which requires an execution time certificate: guaranteeing that the optimization solution is returned before the next sampling time arrives \cite{wu2024direct,wu2024Execution,wu2024time,wu2024parallel,wu2025eiqp} (that is, the execution time is certified to be less than sampling time). Moreover, ultra-low energy consumption and scalability are key in helping \textit{analog optimization} (or \textit{analog MPC}) win the competition with \textit{numerical optimization} (or \textit{numerical MPC}).

\textbf{Discussion (1):} The gradient-based \eqref{eqn_fixed_time_stable_ODE_for_unconstrained} and the proposed Newton-based ODE \eqref{eqn_HMCP_fixed-time-stable_ODE} admit data-dependent and data-independent certified settling times, respectively. While the latter is preferred but at the cost of the analog circuit complexity: \eqref{eqn_fixed_time_stable_ODE_for_unconstrained} scales approximately linearly, whereas \eqref{eqn_HMCP_fixed-time-stable_ODE} scales quadratically with the number of variables $n$ and constraints $m$, respectively. Note that analog circuit implementations of \eqref{eqn_HMCP_fixed-time-stable_ODE} do not require matrix inverse operations, whose corresponding matrix-vector multiplication can be efficiently realized by modern programmable crossbars.

\textbf{Discussion (2):} The theoretically achievable, arbitrarily small settling time of the proposed fixed-time-stable ODE is practically limited by analog hardware constraints like circuit bandwidth, slew rate, and parasitic delays.

\textbf{Limitations. } Although analog hardware has made great progress, such as reconfigurability and programmability \cite{achour2016configuration}, it is still relatively immature compared to digital computers, which is why this article only includes verification via simulation; verification in analog hardware is currently in progress. 

\textbf{Future work. } The coming second part of work, hardware verification, is based on two commercial analog hardware: The Analog Thing and Anadigm FPAA, and three custom PCBs tailored for solving LP, QP, and convex NLP, respectively. Future work also includes developing how to transform combinatorial constrained optimization problems, typically NP-hard, into ODEs and solve them on analog computers. Since analog hardware inevitably experiences device mismatches, finite precision, and thermal/electrical noise that act as bounded disturbances, our future work will focus on establishing the fixed-time property under such perturbations and on developing noise-aware designs supported by Lyapunov-based robustness analysis.

\bibliographystyle{IEEEtran}
\bibliography{ref} 
\end{document}